\documentclass[draft]{amsart}

\usepackage{amssymb,amsmath,tikz}

\newcommand{\Z}{\mathbb{Z}}
\newcommand{\C}{\mathbb{C}}

\newtheorem{theorem}{Theorem}
\newtheorem{proposition}{Proposition}
\newtheorem{lemma}{Lemma}

\theoremstyle{remark}

\newtheorem{remark}{Remark}
\newtheorem{conjecture}{Conjecture}

\title{A simple model of 4d-TQFT}
\author{Rinat Kashaev}
\address{Section de math\'ematiques, Universit\'e de Gen\`eve,
2-4 rue du Li\`evre, 1211 Gen\`eve 4, Suisse\\}
\date{May 22, 2014}
\thanks{Supported in part by Swiss National Science Foundation}

\begin{document}

\begin{abstract} 
We show that, associated with any complex root of unity $\omega$, there exists a particularly simple 4d-TQFT model $M_\omega$  defined on the cobordism category of Delta complexes. For an oriented closed 4-manifold $X$ of Euler characteristic $\chi(X)$,  it is conjectured that the quantity $N^{3\chi(X)/2}M_\omega(X)$, where $N$ is the order of $\omega$, takes only finitely many values as a function of $\omega$. In particular,  it  is equal to 1 for $S^4$, $\left(3+(-1)^{N}\right)/2$ for $S^2\times S^2$, and $N^{-1/2}\sum_{k=1}^N\omega^{k^2}$ for $\C P^2$.
\end{abstract}
\maketitle
\section{Introduction}
\emph{Pachner or bistellar moves} are known to form a finite set of operations on triangulations such that arbitrary triangulations of a piecewise linear (PL) manifold can be related by a finite sequence of Pachner moves \cite{MR1095161,MR1734414}. As a result, the combinatorial framework of triangulated PL manifolds combined with algebraic realizations of Pachner moves can be useful for constructing combinatorial 4-dimensional topological quantum field theories (TQFT) \cite{MR953828,MR1001453}. Realization of this scheme in three dimensions has been initiated in the Regge--Ponzano model~\cite{PonzanoRegge1968}, where the Pachner moves are realized algebraically in terms of the angular momentum $6j$-symbols satisfying the five term Biedenharn--Elliott identity~\cite{MR0052585, Elliott1953}, which has eventually lead to the Turaev--Viro TQFT model~\cite{MR1191386} and subsequent generalizations based on the theory of linear monoidal categories~\cite{MR1292673}. The same scheme in four dimensions is more difficult to realize, mainly because of complicated nature of algebraic constructions generalizing those of the linear monoidal categories though some realizations are known~\cite{MR1273569, MR1295461, MR1706684, MR1931150, MR3116189}. 
In this paper, to any complex root of unity $\omega$, we associate a rather simple model $M_\omega$ of 4d-TQFT defined on the cobordism category of Delta complexes \cite{MR1867354}. The definition is as follows.

We denote by $N\equiv \operatorname{ord}(\omega)$ the order of $\omega$, and we recall that in any Delta complex realizing an oriented $d$-manifold, each $d$-simplex $S$ comes equipped with a sign $\epsilon(S)$ taking the positive value 1 if the orientation induced by the linear order on the vertices of $S$ agrees with the orientation of the manifold.
We specify $M_\omega$ by associating the vector space $\C^N$  to each positive tetrahedron and the dual vector space $\left(\C^N\right)^*$ to each negative tetrahedron.
For a pentachoron (4-simplex) $P$ realizing an oriented 4-ball, we associate the vector 
\begin{equation}
M_\omega(P)\in M_\omega(\partial P)=\otimes_{i=0}^4M_\omega(\partial_iP)\end{equation}
defined by the formula
\begin{equation}
M_\omega(P)=\left\{
\begin{array}{cl}
 Q&\mathrm{if}\ \epsilon(P)=1;\\
 \bar Q& \mathrm{otherwise}.
\end{array}
\right.
\end{equation}
where 
\begin{align}\label{eq:Q}
Q&=\frac1{\sqrt{N}}\sum_{k,l,m\in\Z/N\Z}\omega^{km}e_k\otimes \bar e_{k+l}\otimes e_l\otimes \bar e_{l+m}\otimes e_m,\\
\label{eq:barQ}
\bar Q&=\frac1{\sqrt{N}}\sum_{k,l,m\in\Z/N\Z}\omega^{-km}\bar e_k\otimes e_{k+l}\otimes \bar e_l\otimes e_{l+m}\otimes \bar e_m
\end{align}
 with $\{e_k\}_{k\in\Z/N\Z}$  and $\{\bar e_k\}_{k\in\Z/N\Z}$ being  the canonical dual bases of $\C^N$ and $\left(\C^N\right)^*$ respectively.
 
Let $X$ be an arbitrary Delta complex representing an oriented 4-manifold. We define
\begin{equation}
M_\omega(X)=N^{-|X_0^{\mathrm{int}}|}\operatorname{Ev}(\otimes_{P\in X}M_\omega(P))
\end{equation}
where the tensor product is taken over all pentachora of $X$, 
$\operatorname{Ev}$ is the operation of contracting along all the internal tetrahedra of $X$, and $|X_0^{\mathrm{int}}|$ is the number of internal vertices of $X$. Our main result is the following theorem.
\begin{theorem}\label{thm}
 $M_\omega$ is a well defined 4d-TQFT.
\end{theorem}
The paper is organized as follows. In the next two sections we prove Theorem~\ref{thm} by showing the independence of $M_\omega$ on the branching of Delta triangulations and its invariance under the Pachner moves. In the last section, we provide examples of calculation which hint that the associated invariant despite the simplicity of the model might be interesting.

\section{Behavior under branching changes}
Any Delta triangulation comes equipped with a branching meaning that the vertices of each triangle are linearly ordered. 
\begin{proposition}\label{prop1}
 For any two compact oriented 4-manifold Delta triangulations $X$ and $Y$ differing by a change of branching, one has the equality 
\begin{equation}
M_\omega(X)=b(M_\omega(Y)),\quad b\colon M_\omega(\partial X)\to M_\omega(\partial Y).
\end{equation}
where $b$ is an isomorphism of vector spaces.
\end{proposition}
Let us fix a square root $\sqrt{\omega}$. Following \cite{MR1193849}, we define a function
\begin{equation}
\Phi\colon \Z/N\Z\to\C,\quad \Phi(k)=\left(\sqrt{\omega}\right)^{k(k+N)},
\end{equation}
which has the obvious properties 
\begin{equation}\label{eq:propPhi}
\Phi(k)^2=\omega^{k^2},\quad \Phi(-k)=\Phi(k),\quad \Phi(k+l)=\Phi(k)\Phi(l)\omega^{kl}.
\end{equation}
We also denote
\begin{equation}
\bar\Phi(k)\equiv \frac1{\Phi(k)}.
\end{equation}
Next, we define two vector space isomorphisms
\begin{equation}
S,T\colon \left(\C^N\right)^*\to\C^N,
\end{equation}
by the formulae
\begin{equation}
S\bar e_k=\frac1{\sqrt{N}}\sum_{l\in\Z/N\Z}\Phi(k-l) e_{l},\quad T\bar e_k=\Phi(k) e_{-k}.
\end{equation}
Notice that their inverses are given by the Hermitian conjugate maps :
\begin{equation}
S^{-1}e_k=\bar Se_k=\frac1{\sqrt{N}}\sum_{l\in\Z/N\Z}\bar \Phi(k-l)\bar e_{l},
\quad
T^{-1}e_k=\bar Te_k=\bar \Phi(k) e_{-k}.
\end{equation}
We also define the permutation maps
\begin{equation}
P\colon \left(\C^N\right)^*\otimes \C^N\to \C^N \otimes 
 \left(\C^N\right)^* , \quad \bar P=P^{-1}\colon\C^N \otimes 
 \left(\C^N\right)^*\to  \left(\C^N\right)^*\otimes \C^N.
\end{equation}
The proof of Proposition~\ref{prop1}  is based on the following lemma.
\begin{lemma}\label{lem1} One has the equalities
\begin{multline}\label{eq:sym-rel}
Q=(P\otimes T\otimes \bar T\otimes T)\bar Q=
(T\otimes  \bar P\otimes \bar S\otimes S)\bar Q\\
=(S\otimes\bar S\otimes  P\otimes T)\bar Q
=(T\otimes \bar T\otimes T\otimes   \bar P )\bar Q
\end{multline}
where the vectors $Q$ and $\bar Q$ are defined in \eqref{eq:Q} and \eqref{eq:barQ}.
\end{lemma}
\begin{proof}
 Let us prove the first equality:
 \begin{multline}
\sqrt{N}(P\otimes T\otimes \bar T\otimes T)\bar Q=\sum_{k,l,m\in\Z/N\Z}\omega^{-km}e_{k+l}\otimes \bar e_k\otimes T\bar e_l\otimes \bar Te_{l+m}\otimes T\bar e_m\\=
\sum_{k,l,m\in\Z/N\Z}\omega^{-km}\Phi(l)\bar\Phi(l+m)\Phi(m)e_{k+l}\otimes \bar e_k\otimes e_{-l}\otimes \bar e_{-l-m}\otimes e_{-m}\\=
\sum_{k,l,m\in\Z/N\Z}\omega^{-km-lm}e_{k+l}\otimes \bar e_k\otimes e_{-l}\otimes \bar e_{-l-m}\otimes e_{-m}\\=
\sum_{k,l,m\in\Z/N\Z}\omega^{-km}e_{k}\otimes \bar e_{k-l}\otimes e_{-l}\otimes \bar e_{-l-m}\otimes e_{-m}\\=
\sum_{k,l,m\in\Z/N\Z}\omega^{km}e_{k}\otimes \bar e_{k+l}\otimes e_{l}\otimes \bar e_{l+m}\otimes e_{m}=\sqrt{N}Q
\end{multline}
where in the third equality we used the last property in \eqref{eq:propPhi}, in the forth equality we shifted the summation variable $k\to k-l$, and in the fifth equality we negated the summation variables $l$ and $m$. The other relations are proved in a similar manner.
\end{proof}
\begin{proof}[Proof of Proposition~\ref{prop1}]
For a triangle $f$ of a Delta triangulation, let $C(f)$ be the set of tetrahedra containing $f$. 
Let $X$ and $Y$ be two Delta triangulations differing in the orientation of only one edge $e$. The change of the orientation of $e$ results in changing the sign of each  pentachoron of $X$ containing $e$. By applying the appropriate equality of Lemma~\ref{lem1} to each such pentachoron in $M_\omega(X)$ we observe that for each triangle $f$ containing $e$, there is a cancellation of a pair of $S$ or $T$ operators for each internal tetrahedron of $C(f)$. In this way, we immediately obtain the equality $M_\omega(X)=b(M_\omega(Y))$ where $b$ is given the tensor product of non-canceled $S$ or $T$ operators acting on the boundary tetrahedra. We finish the prove by remarking that any branching change can be obtained as a finite sequence of single edge orientation changes.
\end{proof}
\section{Invariance under the Pachner moves}

A Pachner move in dimension $4$ is associated with a splitting of the boundary of a $5$-simplex into two non-empty disjoint sets of $4$-simplices (pentachora). A Pachner move is called of the \emph{type} $(k,l)$ with $k+l=6$, if the two disjoint subsets of pentachora consist of $k$ and $l$  elements respectively. Thus, altogether, we have Pachner moves of three possible types (3,3), (2,4) and (1,5). Let us discuss in more detail their algebraic realizations in terms polynomial identities for the matrix coefficients of the vectors \eqref{eq:Q} and \eqref{eq:barQ} defined by the formulae:
\begin{equation}\label{eq:qijklm}
Q^{i,j,k}_{l,m}\equiv\langle \bar e_i\otimes e_l\otimes \bar e_j\otimes e_m\otimes \bar e_k,Q\rangle=\frac1{\sqrt{N}}\omega^{ik}\delta_{l,i+j}\delta_{m,j+k}
\end{equation}
and 
\begin{equation}\label{eq:qijklm-bar}
\bar Q_{i,j,k}^{l,m}\equiv\langle  e_i\otimes \bar e_l\otimes e_j\otimes \bar e_m\otimes e_k,\bar Q\rangle=\frac1{\sqrt{N}} \omega^{-ik}\delta_{l,i+j}\delta_{m,j+k}
\end{equation}
\subsection{The type (3,3)}
This is the most fundamental Pachner move as it is the only one which can be written in the form involving only the pentachora of one and the same sign and, in a sense, it implies all other types. 

Consider a 5-simplex with linearly ordered vertices $A=\{v_0,v_1,\dots, v_5\}$. Its boundary is composed of six pentachora $\partial_iA=A\setminus\{v_i\}$ of which three are positive corresponding to even $i$'s and three are negative corresponding to odd $i$'s.  All even (respectively odd) pentachora compose a 4-ball, to be called \emph{even} (respectively \emph{odd}) 4-ball, so that the boundary of both balls are naturally identified as simplicial complexes. Both of these balls, when considered separately, are composed only in terms of positive pentahorons, and the corresponding algebraic condition on the vector $Q$ takes the form
\begin{equation}\label{eq:pachner3-3}
\sum_{s,t,u}Q^{i,l,m}_{s,t}Q^{s,j,n}_{p,u}Q^{t,u,k}_{q,r}=\sum_{s,t,u}Q^{m,n,k}_{s,t}Q^{l,j,t}_{u,r}Q^{i,u,s}_{p,q}
\end{equation}
where the left hand side corresponds to the even 4-ball and the right hand side to the odd one, while the summations in both sides correspond to their own  interior tetrahedra. Namely, denoting the tetrahedron $A\setminus\{v_i,v_j\}$ by $A_{ij}$,  the indices $s,t,u$ correspond to the tetrahedra $A_{02}$, $A_{04}$ and $A_{24}$ in the even 4-ball, and the tetrahedra  $A_{15}$, $A_{35}$ and $A_{13}$ in the odd 4-ball, while the exterior indices $i,j,k,l,m,n,p,q,r$ on both sides correspond to the boundary tetrahedra $A_{01}$, $A_{23}$, $A_{45}$, $A_{03}$, $A_{05}$, $A_{25}$, $A_{12}$, $A_{14}$, $A_{34}$ respectively. All other forms of the Pachner relation of the type (3,3) can be obtained from \eqref{eq:pachner3-3} by applying the symmetry relations~\eqref{eq:sym-rel}.
\begin{lemma}
 The Pachner relation  \eqref{eq:pachner3-3} holds true for the weights \eqref{eq:qijklm}.
\end{lemma}
\begin{proof}
 Ba substituting one after another the explicit forms from \eqref{eq:qijklm}, we have
\begin{multline*}
 N^{3/2}(\mathrm{l.h.s.\ of\ \eqref{eq:pachner3-3}})=\sum_{u}\omega^{im}Q^{i+l,j,n}_{p,u}Q^{l+m,u,k}_{q,r}
 =\omega^{im+(i+l)n}\delta_{p,i+l+j}Q^{l+m,j+n,k}_{q,r}\\=\omega^{im+(i+l)n+(l+m)k}\delta_{p,i+l+j}\delta_{q,l+m+j+n}\delta_{r,j+n+k},
\end{multline*}
and, similarly,
\begin{multline*}
 N^{3/2}(\mathrm{r.h.s.\ of\ \eqref{eq:pachner3-3}})=\sum_{u}\omega^{mk}Q^{l,j,n+k}_{u,r}Q^{i,u,m+n}_{p,q} 
 =\omega^{mk+l(n+k)}\delta_{r,j+n+k}Q^{i,l+j,m+n}_{p,q} \\=\omega^{mk+l(n+k)+i(m+n)}\delta_{r,j+n+k}\delta_{p,i+l+j}\delta_{q,l+j+m+n}.
 \end{multline*}
 Comparing the obtained expressions, we see that they are the same.
\end{proof}
\begin{remark}
 It is interesting to note that by defining three linear maps
\begin{multline}\label{eq:lmr}
L^i,M^j,R^k\colon \C^N\otimes\C^N\to  \C^N\otimes\C^N, \quad Q^{i,j,k}_{l,m}\\=\langle \bar e_j\otimes \bar e_k,L^i(e_l\otimes e_m)\rangle=\langle \bar e_i\otimes \bar e_k,M^j(e_l\otimes e_m)\rangle=\langle \bar e_i\otimes \bar e_j,R^k(e_l\otimes e_m)\rangle,
\end{multline}
we can rewrite the system~\eqref{eq:pachner3-3} as a 3-index family of matrix Yang--Baxter relations in  $\C^N\otimes\C^N\otimes\C^N$:
\begin{equation}
L^i_{12}M^j_{13}R^k_{23}=R^k_{23}M^j_{13}L^i_{12}
\end{equation}
with the standard meaning of the subscripts, for example, $L^i_{12}\equiv L^i\otimes\operatorname{id}_{\C^N}$, etc. It would be interesting to understand the significance of this fact in relationships of 4d-TQFT with lattice integrable models of statistical mechanics.
\end{remark}
\begin{remark}
Another equivalent form of the system~\eqref{eq:pachner3-3} is given by a 3-index family of ``twisted'' pentagon relations either for the $R^i$-matrices
\begin{equation}\label{eq:twist-pent-1}
R^{m}_{12}R^{n}_{13}R^{k}_{23}=\sum_{s,t}Q^{m,n,k}_{s,t}R^{t}_{23}R^{s}_{12}=\frac1{\sqrt{N}}\omega^{mk}R^{n+k}_{23}R^{m+n}_{12},
\end{equation}
or for the $L^i$-matrices
\begin{equation}
L^{m}_{23}L^l_{13}L^{i}_{12}=\sum_{s,t}Q^{i,l,m}_{s,t}L^{s}_{12}L^{t}_{23}=\frac1{\sqrt{N}}\omega^{im}L^{i+l}_{12}L^{l+m}_{23},
\end{equation}
where we use the matrices defined in \eqref{eq:lmr}.
\end{remark}
\subsection{The type (2,4)}
We split  the pentachora of the 5-simplex $A=\{v_0,v_1,\dots, v_5\}$ into a subset of two pentachora  $\partial_1A$ and $\partial_3A$ the complementary subset of other four pentachora. The corresponding algebraic relation takes the form 
\begin{equation}\label{eq:pachner2-4}
\sum_{k,m,n,u,v,w}Q^{i,l,m}_{v,w}Q^{v,j,n}_{p,u}Q^{w,u,k}_{q,r}\bar Q_{m,n,k}^{s,t}=\sum_{u}Q^{l,j,t}_{u,r}Q^{i,u,s}_{p,q},
\end{equation}
and all other forms can be obtained from it by using the symmetry relations~\eqref{eq:sym-rel}.
\begin{lemma}\label{lem2-4}
 The relation \eqref{eq:pachner2-4} holds true  for the weights \eqref{eq:qijklm} and \eqref{eq:qijklm-bar}.
\end{lemma}
\begin{proof}
 We rewrite \eqref{eq:pachner2-4} in the equivalent matrix form
\begin{equation}
\sum_{k,m,n}R^{m}_{12}R^{n}_{13}R^{k}_{23}\bar Q_{m,n,k}^{s,t}=R^{t}_{23}R^{s}_{12}
\end{equation}
and easily prove it by using \eqref{eq:twist-pent-1}:
\begin{multline}
 \sum_{k,m,n}R^{m}_{12}R^{n}_{13}R^{k}_{23}\bar Q_{m,n,k}^{s,t}=\frac1{\sqrt{N}}\sum_{n}\omega^{-(s-n)(t-n)}R^{s-n}_{12}R^{n}_{13}R^{t-n}_{23}\\=N^{-1}\sum_{n}R^{t}_{23}R^{s}_{12}=R^{t}_{23}R^{s}_{12}.
\end{multline}
\end{proof}
\begin{remark}
 As the proof of Lemma~\ref{lem2-4} shows,  the Pachner relation of the type (2,4) given by equation~\eqref{eq:pachner2-4} is clearly weaker than  the Pachner relation of the type (3,3) given by equation~\eqref{eq:pachner3-3}. Namely, we cannot revert the argument of the proof to obtain an equivalence between the two relations.
\end{remark}
\subsection{The type (1,5)}
We split  the pentachora of the 5-simplex $A=\{v_0,v_1,\dots, v_5\}$ into the a set composed of only one pentachoron $\partial_1A$ and the complementary set of other 5 pentachora.
The corresponding algebraic relation takes the form 
\begin{equation}\label{eq:pachner1-5}
N^{-1}\sum_{j,k,l,m,n,r,t,v,w,x}Q^{i,l,m}_{v,w}Q^{v,j,n}_{p,x}Q^{w,x,k}_{q,r}\bar Q_{m,n,k}^{s,t}\bar Q_{l,j,t}^{u,r}=Q^{i,u,s}_{p,q},
\end{equation}
where we have taken into account the fact that vertex $v_1$ is in the interior of the 4-ball corresponding to the left hand side of \eqref{eq:pachner1-5} and, according to our TQFT rules, we have divided the corresponding sum by $N$. As before, all other forms of the Pachner relations of the type (1,5) can be obtained from \eqref{eq:pachner1-5} by using the symmetry relations~\eqref{eq:sym-rel}.
\begin{lemma}
  The relation \eqref{eq:pachner1-5} holds true  for the weights \eqref{eq:qijklm} and \eqref{eq:qijklm-bar}.
\end{lemma}
\begin{proof}
 By using \eqref{eq:pachner2-4}, we write
\begin{multline}
N^{-1}\sum_{j,k,l,m,n,r,t,v,w,x}Q^{i,l,m}_{v,w}Q^{v,j,n}_{p,x}Q^{w,x,k}_{q,r}\bar Q_{m,n,k}^{s,t}\bar Q_{l,j,t}^{u,r}\\
 =N^{-1}\sum_{j,l,r,t,x}Q^{l,j,t}_{x,r}Q^{i,x,s}_{p,q}\bar Q_{l,j,t}^{u,r}=N^{-2}\sum_{j,l,r,t,x}\delta_{x,u}\delta_{x,l+j}\delta_{r,j+t}Q^{i,x,s}_{p,q}\\
=Q^{i,u,s}_{p,q}N^{-2}\sum_{j,l,r,t}\delta_{u,l+j}\delta_{r,j+t}=Q^{i,u,s}_{p,q}N^{-2}\sum_{j,l,t}\delta_{u,l+j}=Q^{i,u,s}_{p,q}N^{-2}\sum_{l,t}1=Q^{i,u,s}_{p,q}.
\end{multline}
\end{proof}

\section{Examples of calculation}
Let us illustrate the calculation of the partition function $M_\omega(S^4)$. by using the standard two pentachora triangulation of $S^4$. 
This is the easiest example to calculate. The 4-sphere can be triangulated in the standard way with two pentachora of opposite signs and  five vertices so that we have
\begin{multline}
M_\omega(S^4)=N^{-5}\sum_{i,j,k,l,m}Q^{i,j,k}_{l,m}\bar Q_{i,j,k}^{l,m}\\=N^{-6}\sum_{i,j,k,l,m}\delta_{j,i+k}\delta_{l,k+m}=N^{-6}\sum_{i,k,m}1=N^{-3}.
\end{multline}
Similar but much more lengthy calculations are needed for other examples. Leaving details to a separate publication, we collect the results into Table~\ref{tab} 
\begin{table}[h]
\[
 \begin{array}{|c|c|c|}
 \hline
  X& \chi(X)&N^{3\chi(X)/2}M_\omega(X)\\
  \hline
  \hline
  S^4   &2&1\\
  \hline
  S^2\times S^2&4&(3+(-1)^N)/2\\
  \hline
  \C P^2&3&N^{-1/2}\sum_{k=1}^N\omega^{k^2}\\
  \hline
  S^3\times S^1&0& 1\\
  \hline
  S^2\times S^1\times S^1&0&(3+(-1)^N)/2 \\
  \hline
  \end{array}
\]
\caption{}
\label{tab}
\end{table}
where $\chi(X)$ is the Euler characteristic. This table is consistent with the following conjecture. 
\begin{conjecture}
 For a given oriented closed 4-manifold $X$, the normalized quantum invariant  $N^{3\chi(X)/2}M_\omega(X)$, considered as a function on the set of all complex roots of unity, takes only finitely many different values.
\end{conjecture}

\def\cprime{$'$} \def\cprime{$'$}

\end{document}